\documentclass[12pt,a4paper]{scrartcl}
\usepackage[utf8]{inputenc}
\usepackage{nameref}
\usepackage[colorlinks=true, urlcolor=blue, citecolor=blue, linkcolor=blue]{hyperref}
\usepackage{amsfonts, amsmath, amssymb, amsthm, thmtools}
\usepackage{cleveref}

\PassOptionsToPackage{usenames,dvipsnames,svgnames}{xcolor}

\declaretheorem[name=Theorem, refname={Theorem, Theorems}, Refname={Theorem, Theorems}, parent=section]{theorem}
\declaretheorem[name=Assumption, refname={assumption, assumptions}, Refname={Assumption, Assumptions}, sibling=theorem]{assumption}

\declaretheorem[name=Definition, refname={definition, definitions}, Refname={Definition, Definition}, sibling=theorem, style=definition]{definition}
\declaretheorem[name=Example, refname={example, examples}, Refname={Example, Examples}, sibling=theorem, style=definition]{example}
\declaretheorem[name=Lemma, refname={lemma, lemmas}, Refname={Lemma, Lemmas}, sibling=theorem]{lemma}

\usepackage{color, enumitem, fancyhdr, float, layout, multirow, setspace, verbatim}
\usepackage{caption, subcaption}
\usepackage[authoryear]{natbib}
\usepackage{authblk}

\usepackage{graphicx}
\graphicspath{{figures/}{../figures/}}

\usepackage[a4paper, margin=2.5cm]{geometry}
\def\Re{\mathbb{R}}

\def\E{\mathbb{E}}

\def\I{\mathbb{I}}
\def\N{\mathbb{N}}
\def\P{\mathbb{P}}

\def\I{{\mathbb I}}

\def\E{{\textbf{E}}}

\def\sP{{\mathcal P}}
\def\P{{\mathbb P}}

\def\Re{{\mathbb R}}

\def\sP{\mathcal{P}}

\def\t0{{\theta_0}}

\title{Posterior Invariance of Multiplicative Contrasts under Margin Constraints in Contingency Tables}
\author[3]{Rafael B. Stern}
\author[1]{Ruobin Gong}
\author[2]{Joseph B. Kadane}
\author[2]{Mark J. Schervish}
\author[2]{Teddy Seidenfeld}

\affil[3]{Department of Statistics, University of S\~{a}o Paulo}
\affil[1]{Department of Statistics, Rutgers University}
\affil[2]{Department of Statistics \& Data Science, Carnegie Mellon University}
\date{\today}

\begin{document}
\maketitle

\begin{abstract}
 \textbf{Abstract}: Measures of association in contingency tables, such as odds ratios and their generalizations, are often studied under different sampling schemes that either fix or leave random the margins of the table. While classical results show that certain odds ratios are unaffected by constraining the margins, it is less clear when this invariance holds more generally. This paper studies posterior inference for a broad class of multiplicative contrasts of multinomial cell probabilities, which we refer to as generalized odds ratios, and addresses exactly when fixing a margin alters inference about them. We consider Bayesian inference under multinomial sampling and under models in which partition sums of the table are fixed in advance, and assume that the marginal and conditional parameters are independent a priori. Under additional mild assumptions, we show that the posterior distribution of a generalized odds ratio is invariant to fixing a margin if and only if the coefficients defining the contrast sum to zero within the margin.
\end{abstract}

\section{Introduction}

A contingency table is a tabular representation of the joint frequency distribution of categorical variables \citep{Agresti2018introduction}. It organizes data into rows and columns, where each cell represents the count of observations that fall into the corresponding combination of categories. \Cref{tab:contingecy} provides a representation of contingency tables for variables $Y$ and $Z$. The right table illustrates the special case of a $2\times2$ contingency table, where both variables are binary.
\begin{figure}[h]
 \centering
 \begin{tabular}{c|llll}
  Y / Z & $z_0$ & $z_1$ & \ldots & $z_{p}$ \\
  \hline
  $y_0$ & $W_{0,0}$ & $W_{0,1}$ &\ldots& $W_{0,p}$ \\
  $y_1$ & $W_{1,0}$ & $W_{1,1}$ &\ldots& $W_{1,p}$ \\
  \ldots & \ldots & \ldots & \ldots & \ldots \\
  $y_{o}$ & $W_{o,0}$ & $W_{o,1}$ &\ldots& $W_{o,p}$
 \end{tabular}\hspace{1cm}%
 \begin{tabular}{c|ll}
  Y / Z & $z_0$ & $z_1$ \\
  \hline
  $y_0$ & $W_{0,0}$ & $W_{0,1}$ \\
  $y_1$ & $W_{1,0}$ & $W_{1,1}$
 \end{tabular}
 \caption{Contingency tables for two categorical variables: $Y$ and $Z$. (left) While $Y$ assumes values in $\{y_0,y_1,\ldots,y_o\}$, 
 $Z$ assumes values in $\{z_0, z_1, \ldots, z_p\}$. 
 The cell count $W_{i,j}$ indicates the number of observations in which $Y=y_i$ and $Z=z_j$. (right) A $2\times 2$ table, in which both
 $Y$ and $Z$ are binary variables.}
 \label{tab:contingecy}
\end{figure}
Contingency tables are particularly useful for examining relationships between categorical variables, enabling researchers to identify patterns, dependencies, or associations \citep{fienberg2007analysis}. They serve as the foundation for a wide variety of methods, such as the $\chi^2$ test of independence \citep{pearson1900x}, Fisher's exact test
\citep{fisher1936design}, tests for homogeneity \citep{stuart1955test}, and log-linear models \citep{bishop2007discrete}.

This paper studies when fixing margins in contingency tables alters Bayesian inference about measures of association. One of the most widely used measures of association between two binary variables is the odds ratio \citep{cornfield1951method}.
Let $\theta_{i,j}$ denote the probability that $Y=y_i$ and $Z=z_j$.
The odds of $Z=z_1$ versus $Z=z_0$ given $Y=y_i$ are
$\theta_{i,1}/\theta_{i,0}$.
The odds ratio is then defined as
\begin{align}
\psi := \frac{\theta_{0,0}\theta_{1,1}}{\theta_{0,1}\theta_{1,0}}.
\end{align}
This measure quantifies the strength and direction of the association between $Y$ and $Z$; in particular, $Y$ and $Z$ are independent if and only if $\psi = 1$.

A natural question is whether the marginal sums of a contingency table provide information about these odds ratios. When observations are independent and have a predefined total, the multinomial distribution is justified for the cell frequencies:
\begin{align}
 (W_{0,0},W_{0,1},W_{1,0},W_{1,1}) \sim 
 \text{Multinomial}(n,\theta_{0,0},\theta_{0,1},\theta_{1,0},\theta_{1,1}).
\end{align}
Under this model, the row sums are jointly ancillary for $\psi$ \citep{fisher1935logic}. 

If the row sums $W^{P_i} := W_{i,0}+W_{i,1}$ are constrained beforehand, the model becomes a product of binomial distributions,
\begin{align}
 W_{i,0} \sim \text{Binomial}(W^{P_i},\theta_{i,0}).
\end{align}
Despite the difference between the multinomial and independent binomial sampling schemes, the information obtained about $\psi$ is the same: for every prior on $\theta$, the posterior distribution of $\psi$ is identical.

However, if both row and column margins are fixed, the sampling distribution of the table is generalized hypergeometric, which forms the basis of Fisher's exact test. This model may contain less information about $\psi$ than the multinomial model \citep{Schervish2025}, reflecting the fact that fixing additional margins can discard information about the association parameter.

This paper studies whether analogous results hold for generalizations of the odds ratios. Specifically, we characterize the conditions under which fixing a margin does not alter the posterior distribution of a broad class of multiplicative contrasts of cell probabilities.
This class encompasses commonly used measures of association.
\begin{example}[Local odds ratio]
The local odds ratio \citep{Rudas1998} generalizes the odds ratio tor larger tables, such as the one shown on the left side of \cref{tab:contingecy}. The local odds ratio for cell $(i,j)$ is defined as
 \begin{align}
  \psi_{i,j} &=
  \frac{\theta_{i,j}\theta_{i+1,j+1}}{\theta_{i+1,j}\theta_{i,j+1}}.
 \end{align}
 One can interpret $\psi_{i,j}$ as the association between $Y$ and $Z$ when these variables are restricted, respectively, to the values $\{y_i,y_{i+1}\}$ and  
 $\{z_j,z_{j+1}\}$. More generally,
 $Y$ and $Z$ are independent when
 $\psi_{i,j} = 1$, for every 
 $0 \leq i \leq {o-1}$ and 
 $0 \leq j \leq p-1$.
\end{example}

\begin{example}[Higher-order odds ratios]
 Higher-order odds ratios \citep{Rudas1998} study the joint dependence of $k$ binary variables, $Y_1,\ldots,Y_k$, arranged in a $2^k$ contingency table.
 In this context, let the cell parameters be $\P(Y_1 = j_1, \ldots, Y_k = j_k|\theta) = \theta_{j_1,\ldots,j_k}$ and,
 for $j = (j_1,\ldots,j_k)$, let
 $|j| = \sum_{i=1}^kj_i \text{ (mod 2)}$.
 The (k-1)-th order odds ratio captures how much the (k-2)-th order odds ratio varies conditional on $Y_k$ and is
 \begin{align}
  \psi^{(k-1)} &=
  \frac{\prod_{j: |j| = 0}\theta_j}{\prod_{j: |j| = 1}\theta_j}.
 \end{align}
\end{example}

We generalize these odds ratios through a class of multiplicative contrasts of the cell probabilities, referred to as generalized odds ratios. We show that the posterior distribution of a generalized odds ratio is invariant to fixing a margin if and only if the corresponding contrast does not depend on the marginal probabilities. Specifically, posterior invariance holds if and only if the coefficients defining the multiplicative contrast sum to zero within each margin of the table.

\Cref{sec:def} introduces the notation and the definition of generalized odds ratios. \Cref{sec:main} presents the main results characterizing when margin constraints affect posterior inference.

\section{Definitions and Notation}
\label{sec:def}

The odds ratios discussed in the previous section can be written as multiplicative contrasts of the cell probabilities. To formalize this idea, it is convenient to represent the contingency table as a vector of cell counts. An arbitrary table is flattened into a single vector of counts, $X = (X_1,\ldots,X_r) \sim \text{Multinomial}(n,\theta)$. The rows of the original contingency table correspond to a partition  $\sP = \{P_1,\ldots,P_k\}$ of the cell indices $\{1,\ldots,r\}$. For each $j \in P_i$, $X_j$ is a cell count in the $i$-th row of the original table. Using this notation, the vector of row sums is defined as $X^\sP := \left(\sum_{i \in P} X_i\right)_{P \in \sP}$.  

\begin{definition}
 \label{def:godds}
 Let $c \in \Re^{r \cdot d}$.
The vector of generalized odds ratios $\psi\in\mathbb{R}^d$ is
 \begin{align} 
  \psi_j := \prod_{i=1}^r \theta_i^{c_{i,j}}.
 \end{align}
\end{definition}
Thus, generalized odds ratios include the odds ratio discussed in the previous section.

\begin{example}
 In a $2 \times 2$ table such that  $X = (W_{0,0},W_{0,1},W_{1,0},W_{1,1})$, $\psi$ equals the usual odds ratio by taking $c_{1} = 1$, $c_{2} = -1$, $c_{3} = -1$, and $c_{4} = 1$. Generalized odds ratios also encompass local and higher-order odds ratios.
\end{example}

To study how fixing partition sums affects inference about generalized odds ratios, it is convenient to work with a reparameterization of $\theta$. Let $\theta^{\sP}_P 
= \sum_{i \in P} \theta_i = \P(X \in P \mid \theta)$. For each $P\in\sP$, define $\nu^P_i = \frac{\theta_i}{\theta^{\sP}_P} = \P(X=i \mid X\in P,\theta)$. There is a one-to-one transformation between $\theta$ and $(\theta^\sP, \nu)$.

Using this reparametrization, we introduce the two sampling schemes considered in this paper. They correspond to whether the partition sums are fixed in advance or not.
\begin{definition}
 The unconstrained model for $X$ is
 $X \sim \text{Multinomial}(n,\theta)$.
 In the constrained model, the partition sums $X^{\sP}$ are fixed in advance. Furthermore, for each $P\in\sP$, let $X_P = (X_i)_{i\in P}$ be the vector of cell counts in $P$. In the constrained model,
 $X_P \sim \text{Multinomial}\left(X^{\sP}_P,\nu^P\right)$, and
 $(X_P)_{P \in \sP}$ are independent.
 The joint densities under the unconstrained and constrained models are $f_u(x,\theta)$ and $f_c(x,\theta)$, respectively.
\end{definition}

\section{Information about the generalized odds ratio in the partition sums}
\label{sec:main}

This section studies how constraining the partition sums of a contingency table affects posterior inference about generalized odds ratios. The main question is whether there exists information about $\psi$ contained in the partition sums $X^\sP$ that is lost when these are fixed in advance. To answer this, it is convenient to work with the reparametrisation $(\theta^\mathcal{P}, \nu)$, which separates the marginal parameters $\theta^\mathcal{P}$
from the within-partition parameters $\nu$. It is common to assume that $\theta^\sP$ is independent of $\nu$ \textit{a priori}. This is the case, for instance, when $\theta$ follows a Dirichlet distribution.

\begin{assumption}
 \label{ass:indep}
 $\theta^{\sP}$ and $\nu$ are independent \textit{a priori}.
\end{assumption}
Under \cref{ass:indep} the constrained and unconstrained models share an important property: the posterior distribution of $\nu$ is identical. Moreover, $\nu$ and $\theta^{\sP}$ remain independent \textit{a posteriori}, as \cref{lemma:equiv} shows:
\begin{lemma} 
 \label{lemma:equiv}
 Under \cref{ass:indep}:
 \begin{enumerate}
  \item The posterior for $\nu$ given $X$
  is the same in the unconstrained and in the
  constrained models,
  \item $\nu$ and $\theta^{\sP}$ are
  conditionally independent given $X$ in
  the unconstrained and in the constrained models.
 \end{enumerate}
\end{lemma}
The proof relies on the fact that the multinomial likelihood separates into a factor depending only on $\theta^{\sP}$ and a factor depending only on $\nu$. Under the unconstrained model, both factors update the posterior; under the constrained model, $X^{\sP}$ is fixed, so only the $\nu$ factor contributes to the likelihood, leaving $\theta^{\sP}$ at its prior.

\Cref{lemma:equiv} implies that the posterior distribution of $\nu$ contains exactly the same information in the constrained and unconstrained models. Consequently, if a generalized odds ratio $\psi$ depends only on $\nu$, its posterior distribution must be identical in both models. Differences can arise only when $\psi$ also depends on $\theta^{\sP}$. In order to better distinguish this condition, \cref{lemma:log_odds} below shows that
$\psi_j$ can be decomposed into a term that depends solely on $\nu$ and another that depends solely on $\theta^{\sP}$.

\begin{definition}
 Let $\tau_j := 
 \log \left(\prod_{P \in \sP} 
  (\theta^{\sP}_P)^{\sum_{i \in P} c_{i,j}}\right)$ and $\rho_j := \log \left(\prod_{P \in \sP}\prod_{i \in P} 
 (\nu^P_i)^{c_{i,j}}\right)$.
\end{definition}

\begin{lemma}
 \label{lemma:log_odds}
 $\log\psi_j = \tau_j + \rho_j$.
\end{lemma}

\begin{proof}
  \begin{align}
  \log\psi_j 
  &= \log\left(\prod_{P \in \sP}\prod_{i \in P} 
  (\theta^{\sP}_P \cdot \nu^P_i)^{c_{i,j}}\right) \nonumber \\
  &= \log\left(\prod_{P \in \sP} 
  (\theta^{\sP}_P)^{\sum_{i \in P} c_{i,j}}\right) +
  \log\left(\prod_{P \in \sP}\prod_{i \in P} 
  (\nu^P_i)^{c_{i,j}}\right)
  = \tau_j + \rho_j
 \end{align}
\end{proof}
An immediate consequence of \Cref{lemma:log_odds} is that, if $\sum_{i \in P} c_{i,j} = 0$, for every $j$ and $P$, then
$\tau \equiv 0$ and
$\psi$ depends only on $\nu$. In this case, \cref{lemma:equiv} implies that, under \cref{ass:indep}, the posterior of $\psi$ is identical in the constrained and unconstrained models.

In contrast, \Cref{ass:np_large_n} presents mild conditions under which the posterior for $\psi$ fails to be equivalent in the constrained and unconstrained models.

\begin{assumption} \
 \label{ass:np_large_n}
 \begin{enumerate}
  \item $f(\theta) > 0$, for every $\theta$. That is, $\theta$ has full support on the simplex,
  \item There exists $j^*$ such that:
  \begin{enumerate}
   \item $\sum_{i \in P}c_{i,j^*} \neq 0$, for some $P \in \sP$,
   \item For every $x$,
   the characteristic function of $\rho_j$ under the unconstrained model, $\phi^u_{\rho_{j^*}|X=x}(t)$, is such that $\{t: \phi_{\rho_{j^*}|X=x}(t) \neq 0\}$ is dense, and
   \item Let $K_+ = \{P \in \sP: \sum_{i \in P} c_{i,j^*} > 0\}$ and
   $K_- = \{P \in \sP: \sum_{i \in P} c_{i,j^*} < 0\}$. 
   Either $|K_+| > 0$ and $n \geq |K_+|$ or $|K_-| > 0$ and $n \geq |K_-|$.
  \end{enumerate}
 \end{enumerate}
\end{assumption}

\Cref{ass:np_large_n}.1 rules out degenerate cases by ensuring that $\psi$ is not supported on a lower-dimensional subset of the parameter space. \Cref{ass:np_large_n}.2.a guarantees that $\psi$ depends nontrivially on both components of the reparametrization, $\nu$ and $\theta^{\sP}$, so that variation in $\psi$ reflects information from each. \Cref{ass:np_large_n}.2.b ensures that the contribution of $\tau_{j^*}$ can be identified from the decomposition $\log\psi_{j^*} = \tau_{j^*} + \rho_{j^*}$; intuitively, it prevents $\rho_{j^*}$ from masking the effect of $\tau_{j^*}$. This requirement is satisfied, for instance, when the distribution of $\rho_{j^*}$ is sufficiently light-tailed. Finally, \Cref{ass:np_large_n}.2.c imposes that $n$ is large enough relative to the dimension of the parameter space. In particular, this condition holds whenever $n \geq r$.

Based on the above discussion,
\Cref{thm:equiv} provides general conditions that ensure whether the posterior for $\psi$ is equivalent in the constrained and unconstrained models.

\begin{theorem}
 \label{thm:equiv}
 Under \cref{ass:indep},
 \begin{enumerate}
  \item If $\sum_{i \in P}c_{i,j} = 0$, 
  for every $P \in \sP$ and $1 \leq j \leq d$, then the posterior distribution of $\psi$ given $X$
  is the same in the constrained and unconstrained models.
  \item If $\sum_{i \in P}c_{i,j} \neq 0$, for some $P \in \sP$ and $1 \leq j \leq d$, then
  under \cref{ass:np_large_n}, there exists $x$ such that the posterior distribution of $\psi$ given $X=x$ differs in the constrained and unconstrained models.
 \end{enumerate}
\end{theorem}

The proof of \cref{thm:equiv} is in the Appendix. The first part holds because, when the coefficients sum to zero within each partition, \cref{lemma:log_odds} gives $\tau \equiv 0$, so $\psi$ reduces to a function of $\nu$ alone, and the conclusion follows from \cref{lemma:equiv}. The second part is established by comparing the characteristic functions of $\log\psi_{j^*}$ under both models: since $\tau_{j^*}$ is updated by $X$ in the unconstrained model but not in the constrained one, the two characteristic functions can be shown to differ under \cref{ass:np_large_n}.

Given its conjugacy, the Dirichlet prior is commonly used in Multinomial models. The following example applies \cref{thm:equiv} to a Dirichlet prior.

\begin{example}[Dirichlet prior]
 \label{ex:dirichlet}
 Let $\theta \sim \text{Dirichlet}(\alpha)$.
 It follows from the
 aggregation and neutrality
 properties of the Dirichlet,
 that  $\theta^{\sP}$ and 
 $\nu$ are independent.
 Hence,
 it follows from \cref{thm:equiv}.1 that,
 if $\sum_{i \in P}c_i = 0$ for every $P \in \sP$, then
 the posterior for $\psi$ given $X$ is the same in the constrained and unconstrained models.

 Next, consider that $\sum_{i \in P}c_i \neq 0$ for some $P$.
 We proceed to derive $\phi^u_{\log\psi|X=x}$ and $\phi^c_{\log\psi|X=x}$. It follows from \cref{lemma:equiv} that $\phi^c_{\rho|X=x} \equiv \phi^u_{\rho|X=x}$. Also, since $\nu^P|X=x \sim \text{Dirichlet}(\alpha_P+X_P)$,
 \begin{align}
  \phi^c_{\rho|X=x}(t) = \phi^u_{\rho|X=x}(t) 
  &= \prod_{P \in \sP}\E\left[\prod_{i \in P}(\nu^P_i)^{itc_i}\bigg|X=x\right] \nonumber \nonumber \\
  &= \prod_{P \in \sP}
  \left(\frac{\Gamma(\sum_{i \in P} (\alpha_i + x_i))}{\Gamma(\sum_{i \in P} (\alpha_i + x_i + itc_i))} \cdot 
  \prod_{i \in P}
  \frac{\Gamma(\alpha_i + x_i + itc_i)}{\Gamma(\alpha_i + x_i)}\right).
 \end{align}
 Since the $\Gamma$ function has no zeroes and the Dirichlet distribution has full support, it follows that \cref{ass:np_large_n}.1, \cref{ass:np_large_n}.2.a, and
 \cref{ass:np_large_n}.2.b are satisfied.

 Next, we evaluate $\phi^u_{\tau|X=x}$ and $\phi^c_{\tau|X=x}$. Let 
 for each $z \in \Re^r$, 
 $z^{\mathcal{P}} = (\sum_{i \in P}z_i)_{P \in \mathcal{P}}$ be the vector sums within each partition. Under the unconstrained model,
 $\theta^{\sP}|X=x \sim \text{Dirichlet}(\alpha^\sP + x^\sP)$. Similarly, under the constrained model, $\theta^{\sP}|X=x \sim \text{Dirichlet}(\alpha^\sP)$.
 Therefore,
 \begin{align}
  \phi^u_{\tau|X=x}(t)
  &= \E\left[\prod_{P \in \sP} 
  (\theta^{\sP}_P)^{itc^\sP_P}\bigg|X=x \right]
  = 
  \frac{\Gamma\left(n+\sum_{P\in\sP}\alpha^\sP_P\right)}
       {\Gamma\left(n+\sum_{P\in\sP}(\alpha^\sP_P+itc^\sP_P)\right)}
  \prod_{P\in\sP}
  \frac{\Gamma(\alpha^\sP_P+x^\sP_P+itc^\sP_P)}
       {\Gamma(\alpha^\sP_P+x^\sP_P)}, \nonumber \\
  \phi^c_{\tau|X=x}(t)
  &= \E\left[\prod_{P \in \sP} 
  (\theta^{\sP}_P)^{itc^\sP_P}\right] =
  \frac{\Gamma\!\left(\sum_{P\in\sP}\alpha^\sP_P\right)}
       {\Gamma\!\left(\sum_{P\in\sP}(\alpha^\sP_P+itc^\sP_P)\right)}
  \prod_{P\in\sP}
  \frac{\Gamma(\alpha^\sP_P+itc^\sP_P)}
       {\Gamma(\alpha^\sP_P)}.
 \end{align}
 Since $\log\psi = \tau + \rho$ and $\tau$ and $\rho$ are independent under both sampling schemes,
 \begin{align}
 \phi^u_{\log\psi|X=x}(t)
 &= \phi^u_{\rho|X=x}(t) \cdot \frac{\Gamma\left(n+\sum_{P\in\sP}\alpha^\sP_P\right)}
       {\Gamma\left(n+\sum_{P\in\sP}(\alpha^\sP_P+itc^\sP_P)\right)}
  \prod_{P\in\sP}
  \frac{\Gamma(\alpha^\sP_P+x^\sP_P+itc^\sP_P)}
       {\Gamma(\alpha^\sP_P+x^\sP_P)}, \nonumber \\
 \phi^c_{\log\psi|X=x}(t) 
 &= \phi^u_{\rho|X=x}(t) \cdot
  \frac{\Gamma\!\left(\sum_{P\in\sP}\alpha^\sP_P\right)}
       {\Gamma\!\left(\sum_{P\in\sP}(\alpha^\sP_P+itc^\sP_P)\right)}
  \prod_{P\in\sP}
  \frac{\Gamma(\alpha^\sP_P+itc^\sP_P)}
       {\Gamma(\alpha^\sP_P)}.
 \end{align}
 The above characteristic functions are equal if $n = 1$, $r = 4$, $\sP = \{\{1,2\},\{3,4\}\}$, $\alpha \equiv 1$, and $c \equiv 0.5$,
 \begin{align*}
 \phi^u_{\log\psi|X=x}(t) 
 &= \phi^u_{\rho|X=x}(t) \cdot \frac{\Gamma\left(5\right)}
       {\Gamma\left(5+2it\right)} \cdot \frac{\Gamma(2+it)}{\Gamma(2)} \cdot \frac{\Gamma(3+it)}{\Gamma(3)} \\
       &= \phi^u_{\rho|X=x}(t) \cdot \frac{\Gamma\left(4\right)}
       {\Gamma\left(4+2it\right)} \cdot \frac{\Gamma(2+it)}{\Gamma(2)} \cdot \frac{\Gamma(2+it)}{\Gamma(2)}, \\
 \phi^c_{\log\psi|X=x}(t)
 &= \phi^u_{\rho|X=x}(t) \cdot \frac{\Gamma\left(4\right)}
       {\Gamma\left(4+2it\right)} \cdot \frac{\Gamma(2+it)}{\Gamma(2)} \cdot \frac{\Gamma(2+it)}{\Gamma(2)}.
 \end{align*}
 As long as \cref{ass:np_large_n}.2.c is also satisfied, then there exists $x$ such that $\phi^u_{\log\psi|X=x} \neq \phi^c_{\log\psi|X=x}$ and the posterior for $\psi$ is different in the constrained and unconstrained models. Since
 $c \equiv 0.5$, obtain that $|K^{+}| = 2$, and 
 \cref{ass:np_large_n}.2.c is satisfied as long as $n \geq 2$. 
 
 This example shows that the condition $\sum_{i\in P} c_i \neq 0$ alone does not guarantee different posterior distributions. However, as long as Assumption~\ref{ass:np_large_n}.2.c is also satisfied, then there exists $x$ such that $\phi^u_{\log\psi|X=x} \neq \phi^c_{\log\psi|X=x}$ and therefore the posterior distributions differ.
\end{example}

\Cref{thm:equiv,lemma:equiv} rely on the independence between $\nu$ and $\theta^{\sP}$. Without this assumption, the posterior distribution of $\nu$ may differ between the constrained and unconstrained models. As a result, the posterior for $\psi$ might differ in these models even when $\sum_{i \in P}c_{i,j} = 0$ for every $P$ and $j$, as illustrated in the following example.

\begin{example}[Dependence between $\theta^\sP$ and $\nu$]
 Let $r = 4$, $\sP=\{\{1,2\},\{3,4\}\}$,
 $\theta^{\sP}_{\{1,2\}} \sim \text{Beta}(\alpha_1,\alpha_2)$, 
 $\P(\theta^{\sP} = \nu^{\{1,2\}}) = 1$, and
 $\nu_3|\theta^{\sP}_{\{3,4\}} \sim \text{Beta}(\alpha_3,\alpha_4)$.
 In both the unconstrained and constrained models,
 \begin{align}
  \nu^{\{3,4\}}|X \sim \text{Beta}(\alpha_3+x_3, \alpha_4+x_4).
 \end{align}
 
 However, while in the unconstrained model 
 $\nu_1|X \sim \text{Beta}(\alpha_1+2x_1+x_2,\alpha_2+x_2+x_3+x_4)$, in the constrained model
 $\nu_1|X \sim \text{Beta}(\alpha_1+x_1,\alpha_2+x_2)$. Hence, by taking the usual odds ratio, with $c_1 = c_4 = 1$ and $c_2 = c_3 = -1$, it follows from \cref{ex:dirichlet} that 
 \begin{align}
 \phi^u_{\log\psi|X=x}(t)
 &= \frac{\Gamma(\alpha_3 + x_3 - it)}{\Gamma(\alpha_3 + x_3)} \cdot\frac{\Gamma(\alpha_4 + x_4 + it)}{\Gamma(\alpha_4 + x_4)} \cdot \frac{\Gamma(\alpha_1 + 2x_1 + x_2 + it)}{\Gamma(\alpha_1 + 2x_1 + x_2)} \cdot\frac{\Gamma(\alpha_2 + x_2+x_3+x_4 - it)}{\Gamma(\alpha_2 + x_2+x_3+x_4)}, \nonumber \\
 \phi^c_{\log\psi|X=x}(t) 
 &= \frac{\Gamma(\alpha_3 + x_3 - it)}{\Gamma(\alpha_3 + x_3)} \cdot\frac{\Gamma(\alpha_4 + x_4 + it)}{\Gamma(\alpha_4 + x_4)} \cdot \frac{\Gamma(\alpha_1 + x_1 + it)}{\Gamma(\alpha_1 + x_1)} \cdot\frac{\Gamma(\alpha_2 + x_2 - it)}{\Gamma(\alpha_2 + x_2)}.
 \end{align}
 If for instance, $\alpha \equiv 1$ and $x \equiv 1$,
 \begin{align}
 \phi^u_{\log\psi|X=x}(t)
 &= \Gamma(2 - it)^2 \cdot \Gamma(2+it)^2 \cdot 
 \frac{(9+t^2)(4+t^2)}{36}, \nonumber \\
 \phi^c_{\log\psi|X=x}(t) 
 &= \Gamma(2-it)^2\Gamma(2+it)^2.
 \end{align}
 Hence, when $\theta^\sP$ is dependent of $\nu$,
 the posterior of $\psi$ can differ between models even though $\sum_{i \in P}c_i = 0$, for every $P \in \sP$,
 because \cref{lemma:equiv} no longer applies and the posterior for $\nu$ might differ between models. 
\end{example}

The next example illustrates that, as long as $\rho$ is light-tailed, \cref{ass:np_large_n}.2.b is satisfied:

\begin{example}[Light-tailed $\rho$]
\label{ex:analytic}
Assume that, for every $x$,
\begin{align}
 \limsup_{k\rightarrow\infty} 
\frac{(\E[|\rho|^k|X=x])^{1/k}}{k} < \infty .
\end{align}
Then it follows from \citet{feller1991introduction}[p.514] that the characteristic function
$\phi^u_{\rho|X=x}(t)$ admits an analytic extension to the complex plane,
denoted $\phi^{u,*}_{\rho|X=x}(t)$.
Since $\phi^{u,*}_{\rho|X=x}(0)=1$ and $\phi^{u,*}_{\rho|X=x}$ is analytic,
its set of zeros 
$\{t \in \mathbb{C} : \phi^{u,*}_{\rho|X=x}(t)=0\}$ is isolated.
Since $\phi^{u,*}_{\rho|X=x}$ extends $\phi^u_{\rho|X=x}$,
it follows that
$\{t \in \mathbb{R} : \phi^u_{\rho|X=x}(t)=0\}
$ is also an isolated set.
Therefore, as long as the moments of $\rho$ do not grow too quickly, Assumption~\ref{ass:np_large_n}.2.b is satisfied.

Note that Assumption~\ref{ass:np_large_n}.2.b might be satisfied even when $\rho$ is heavy-tailed. For instance, if $\rho \sim \text{Cauchy}(0,1)$, then $\phi_{\rho}(t) = \exp(-|t|)$, which satisfies Assumption~\ref{ass:np_large_n}.2.b.
\end{example}

\Cref{thm:equiv} addresses posterior equivalence in finite samples. It is also instructive to examine how the constrained and unconstrained models differ asymptotically under \cref{ass:np_large_n}.

Under the unconstrained model, since $X \sim \text{Multinomial}(n,\theta)$,  for any prior with full support, the posterior for $\theta$ concentrates at the true value $\theta_0$ as $n \rightarrow \infty$. Consequently, since $\psi$ is a continuous function of $\theta$, the posterior distribution of $\psi$ converges to a point mass at $\psi(\theta_0)$.

The situation is fundamentally different under the constrained model. In this case, the partition sums $X^{\sP}$ are fixed by design, and therefore no learning about $\theta^{\sP}$ occurs, even asymptotically. Although $\nu$ is consistently learned from the within-partition counts, the posterior distribution of $\theta^{\sP}$ remains equal to its prior distribution for all sample sizes. As a result, whenever $\psi$ depends nontrivially on $\theta^{\sP}$, its posterior distribution cannot concentrate to a point mass under the constrained model.

This asymptotic discrepancy highlights that posterior non-equivalence between constrained and unconstrained sampling schemes is not merely a finite-sample phenomenon. When an odds ratio involves marginal parameters, constraining the corresponding partition sums prevents asymptotic learning about those components, leading to persistent posterior uncertainty even as $n \rightarrow \infty$. In contrast, odds ratios that depend only on $\nu$ exhibit posterior concentration in both models, consistent with the finite-sample equivalence established in \cref{thm:equiv}. 

The results in \cref{thm:equiv} characterize posterior invariance when a single partition is constrained. A natural question is whether analogous results hold when two partitions are constrained simultaneously. The following example shows that this is not the case in general, and that constraining a second margin can alter the posterior of $\psi$ even when each partition satisfies \cref{ass:np_large_n}.

\begin{example}[Simultaneous partition constraint]
 Consider a flattened $2 \times 2$ table with
  observed counts $x=(7,1,1,1)$ and
 $\theta \sim \text{Dirichlet}(1,1,1,1)$. By \cref{thm:equiv}, fixing either the row or the column margins alone does not alter the posterior of the odds ratio, $\psi = \frac{\theta_1 \theta_4}{\theta_2 \theta_3}$.

 However, when both margins are fixed simultaneously, the sampling distribution of the table is Fisher's noncentral hypergeometric, and the posterior of $\psi$ is no longer invariant. Indeed, \cref{fig:double_constraint} displays the different posteriors obtained for $\log\psi$ under single-partition and double-partition constraints.

 \Cref{thm:equiv} does not generalize to this example because the reparametrization $(\theta^\sP,\nu)$ is partition-dependent and, hence, cannot be used for both partitions simultaneously.

 \begin{figure}
  \centering
  \includegraphics[width=0.5\linewidth]{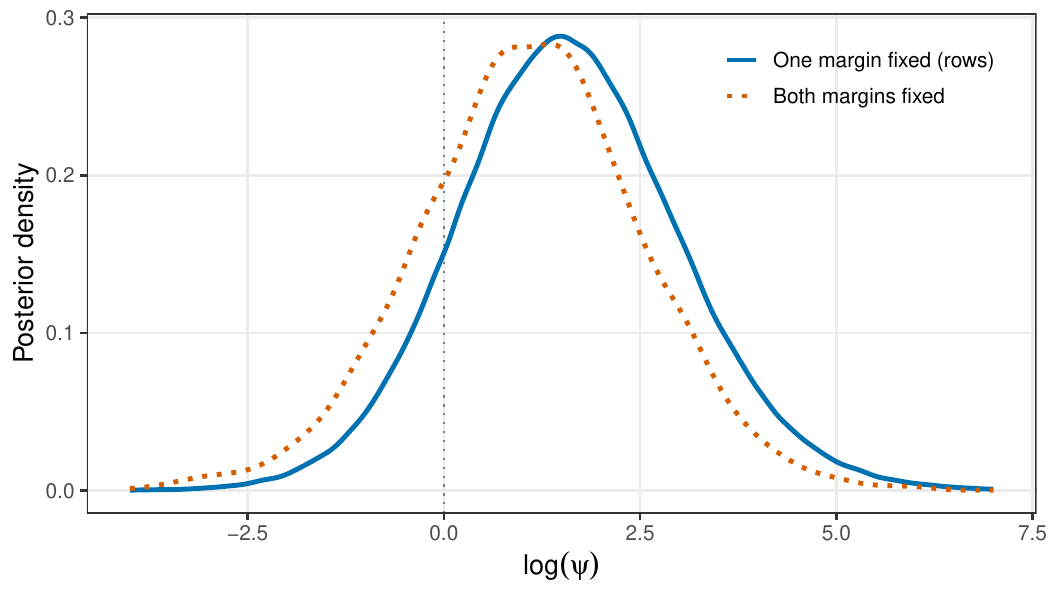}
  \caption{Comparison between the posteriors for the log odds ratio in a $2 \times 2$ when $x = (7,1,1,1)$ under different sampling schemes: (solid blue) row sums are fixed, and (dotted orange) both row and column sums are fixed.}
  \label{fig:double_constraint}
 \end{figure}
\end{example}

\section{Conclusion}

This paper studied when fixing the margins of a contingency table alters Bayesian inference about measures of association. We considered a broad class of multiplicative contrasts of cell probabilities, called generalized odds ratios, and characterized the conditions under which their posterior distribution is invariant to constraining the partition sums of the table.

The main result, \cref{thm:equiv}, shows that under mild conditions this invariance holds if and only if the contrast coefficients sum to zero within each partition. This condition unifies and extends classical results for the standard odds ratio in $2\times 2$ tables, and applies broadly to local odds ratios, and higher-order odds ratios. The distinction between constrained and unconstrained models is not merely a finite-sample phenomenon: when a generalized odds ratio depends nontrivially on $\theta^\sP$, the constrained model cannot learn about those components even as the sample size increases, leading to persistent posterior uncertainty.

\section*{Acknowledgments}

Rafael Bassi Stern is grateful for the financial support of CNPq (grant 313557/2025-0) and University of São Paulo (PRPI/USP 58/2023), and produced this work as part of the activities of FAPESP, Brazil Research, Innovation and Dissemination Center for Neuromathematics (grant 2013/07699-0).

\bibliographystyle{plainnat}
\bibliography{main}

\section{Proofs}

\subsection*{Proof of \cref{lemma:equiv}}

\begin{lemma}[\citet{Kadane2021}]
 \label{lemma:likelihood_ci}
 $X \sim \text{Multinomial}(n,\theta)$ 
 if and only if:
 \begin{itemize}
  \item $X^{\sP} \sim 
  \text{Multinomial}(n, \theta^{\sP})$,
  \item $X_P |X^{\sP} \sim 
  \text{Multinomial}\left(X^{\sP}_P,\nu^P\right)$, and
  \item $(X_P)_{P \in \sP}$ are conditionally
  independent given $X^{\sP}$.
 \end{itemize}
 Furthermore,
 \begin{align}
  f(x|\theta^{\sP},\nu)
  &= \underbrace{\left(n! \prod_{P \in \sP}
  \frac{(\theta^\sP_P)^{x^{\sP}_P}}{(x^{\sP}_P)!}\right)}_{\text{Multinomial}(n,\theta^{\sP})} 
  \ \cdot \
  \prod_{P \in \sP}
  \underbrace{(x^{\sP}_P)!
  \prod_{i \in P} \frac{ (\nu^P_i)^{x_i}}{x_i!}}_{\text{Multinomial}(x^{\sP}_P,\nu^P)}.
 \end{align}
\end{lemma}

\begin{proof}[Proof of \cref{lemma:equiv}]
 \begin{align}
  f_u(\theta^{\sP},\nu|x) 
  &\propto f_u(\theta^{\sP},\nu,x) \\
  &= f(\nu)f(\theta^{\sP}) f_u(x|\nu, \theta^{\sP}) \\
  &= f(\nu)f(\theta^{\sP}) 
  \left(n! \prod_{P \in \sP}
  \frac{(\theta^\sP_P)^{x^{\sP}_P}}{(x^{\sP}_P)!}\right)
  \cdot 
  \prod_{P \in \sP}
  (x^{\sP}_P)!
  \prod_{i \in P} \frac{ (\nu_P)_i^{x_i}}{x_i!}  
  &\text{\cref{lemma:likelihood_ci}} \\
  &\propto \underbrace{
  \left(f(\theta^{\sP}) \prod_{P \in \sP} 
  (\theta^\sP_P)^{x^{\sP}_P}\right)}_{\propto f_u(\theta^{\sP}|x)} \cdot 
  \underbrace{\left(f(\nu) 
  \prod_{P \in \sP}\prod_{i \in P} (\nu_P)_i^{x_i}\right)}_{\propto f_u(\nu|x)} \\
  & \\
  f_c(\theta^{\sP},\nu|x) 
  & \propto f_c(\theta^{\sP},\nu,x) \\
  &\propto
  f(\theta^{\sP})f(\nu) 
  \I(X^{\sP}) \prod_{P \in \sP} (x^{\sP}_P)!
  \prod_{i \in P} \frac{ (\nu_P)_i^{x_i}}{x_i!} \\
  &\propto \underbrace{f(\theta^{\sP})}_{\propto f_c(\theta^{\sP}|x)} \cdot
  \underbrace{\left(f(\nu) 
  \prod_{P \in \sP}\prod_{i \in P} (\nu_P)_i^{x_i}\right)}_{\propto f_c(\nu|x)}
 \end{align}
\end{proof}

\subsection*{Proof of \cref{thm:equiv}}

\begin{lemma}
 \label{lemma:plain}
 $\tau_j$ is independent of $X$ if and only if, for every
 $\alpha \in \N^k$ such that $\sum_{i=1}^k \alpha_i = n$,
 \begin{align}
  \E\left[\prod_{i=1}^k \theta_i^{\alpha_i}|\tau_j \right] \equiv 
 \E\left[\prod_{i=1}^k \theta_i^{\alpha_i}\right].
 \end{align}
\end{lemma}

\begin{proof}
 For every
 $\alpha \in \N^k$ such that $\sum_{i=1}^k \alpha_i = n$,
 \begin{align}
  \P(X = \alpha) 
  &= \E[\P(X=\alpha|\theta)] 
  = \E\left[\prod_{i=1}^k \theta_i^{\alpha_i}\right] 
  \text{, and } \\
  \P(X=\alpha|\tau_j)
  &= \E[\P(X=\alpha|\tau_j,\theta)|\tau_j] 
  = \E\left[\prod_{i=1}^k \theta_i^{\alpha_i}|\tau_j\right]
 \end{align}
 The proof is complete by noting that
 $X$ is independent of $\tau_j$ if and only if, for every
 $\alpha \in \N^k$ such that $\sum_{i=1}^k \alpha_i = n$,
 $\P(X=\alpha|\tau_j) \equiv \P(X = \alpha)$.
\end{proof}

\begin{lemma}
 \label{lemma:notindep}
 Let $K_+ = \{i: \sum_{i \in P} c_{i,j^*} > 0\}$ and
 $K_- = \{i: \sum_{i \in P} c_{i,j^*} < 0\}$. If
 $f(\theta) > 0$:
 \begin{enumerate}[label=(\alph*)]
  \item If $|K_+| > 0$ and $n \geq |K_+|$, then $X$ is not independent of $\tau_{j^*}$,
  \item If $|K_-| > 0$ and $n \geq |K_-|$, then
 $X$ is not independent of $\tau_{j^*}$.
 \end{enumerate}
\end{lemma}

\begin{proof} 
 \begin{enumerate}[label=(\alph*)]
  \item Let $h(\theta) = \prod_{P \in K_+} \theta^{\sP}_P$ and
  $C := \max_{P \in K_+} \sum_{i \in P} c_{i,j^*}$:
  \begin{align}
  \tau_{j^*} 
  &= \log \left(\prod_{P \in \sP} 
  (\theta^{\sP}_P)^{\sum_{i \in P} c_{i,j^*}}\right) \\
  &\geq \log\left(\prod_{P \in K_+} 
  (\theta^{\sP}_P)^{\sum_{i \in P} c_{i,j^*}}\right) \\
  &\geq \log\left(h(\theta)^C\right)
 \end{align}
 Hence, $\tau_{j^*} < M$ implies that $h(\theta) < \exp(M)^{C^{-1}}$, that is,
 $\E[h(\theta)|\tau_{j^*} < M] \rightarrow 0$ as $M \rightarrow -\infty$. Also,
 since $\theta$ has full support,
 $\E[h(\theta)] > 0$. Conclude from 
 \cref{lemma:plain} that 
 $X$ is not independent of $\tau_{j^*}$.
 
 \item Let $h(\theta) = \prod_{P \in K_-} \theta^{\sP}_P$ and
  $C := \min_{P \in K_-} \sum_{i \in P} c_{i,j^*}$:
  \begin{align}
  \tau_{j^*} 
  &= \log \left(\prod_{P \in \sP} 
  (\theta^{\sP}_P)^{\sum_{i \in P} c_{i,j^*}}\right) \\
  &\leq \log\left(\prod_{P \in K_-} 
  (\theta^{\sP}_P)^{\sum_{i \in P} c_{i,j^*}}\right) \\
  &\leq \log\left(h(\theta)^C\right)
 \end{align}
 Hence, $\tau_{j^*} > M$ implies that $h(\theta) < \exp(M)^{C^{-1}}$, that is,
 $\E[h(\theta)|\tau_{j^*} > M] \rightarrow 0$ as $M \rightarrow \infty$. Also,
 since $\theta$ has full support,
 $\E[h(\theta)] > 0$. Conclude from 
 \cref{lemma:plain} that 
 $X$ is not independent of $\tau_{j^*}$.
 \end{enumerate}
\end{proof}

\begin{proof}[Proof of \cref{thm:equiv}]
 If $c \in \Re^{r \times d}$ is such that,
 for every $P \in \sP$ and $1 \leq j \leq d$,
 $\sum_{i \in P}c_{i,j} = 0$, then
 $\tau \equiv 0$ and
 it follows from \cref{lemma:log_odds} that 
 $\psi_j = \exp(\rho_j)$ is a function of $\nu$.
 It follows directly from 
 \cref{lemma:equiv} that
 $\psi$ has the same distribution in
 the unconstrained and constrained models.

 Next, assume that $c \in \Re^{r \times d}$ is such that, for some $P \in \sP$ and $1 \leq j^* \leq d$,
 $\sum_{i \in P}c_{i,j^*} \neq 0$.
 Let $\phi^u$ and $\phi^c$ denote
 characteristic functions under
 the unconstrained and constrained models.
 It follows from \cref{lemma:equiv,lemma:log_odds} that
 \begin{align}
  \phi^u_{\log\psi_{j^*}|X}(t)
  &= \phi^u_{\tau_{j^*}|X}(t) \cdot
  \phi^u_{\rho_{j^*}|X}(t), 
  \text{ and} \\
  \phi^c_{\log\psi_{j^*}|X}(t)
  &= \phi^u_{\tau_{j^*}}(t) \cdot
  \phi^u_{\rho_{j^*}|X}(t).
 \end{align}
 Since $\phi^u_{\rho_{j^*}|X}(t)$
 is continuous and
 $\{t: \phi^u_{\rho_{j^*}|X}(t) \neq 0\}$ is dense, it is sufficient to show that
 there exist $t$ and $x$ such that
 $\phi^u_{\tau_{j^*}|X}(t) \neq \phi^u_{\tau_{j^*}}(t)$,
 that is, to show that
 $\tau_j^*$ is not independent of $X$.
 The latter follows from \cref{lemma:notindep}.
\end{proof}

\end{document}